\documentclass[reqno,a4paper,11pt]{amsart}

\numberwithin{equation}{section}

\usepackage[utf8]{inputenc}
\usepackage{url}
\usepackage{amsmath, amsthm, amssymb, amscd, accents, bm}
\usepackage{mathtools}
\usepackage{mdwlist}
\usepackage{paralist}
\usepackage{datetime}
\usepackage{hyperref}
\usepackage{colonequals}
\usepackage{mathtools}
\usepackage[sort,nocompress]{cite}

\mathtoolsset{showonlyrefs}

\newtheorem{definition}{Definition}[section]
\newtheorem{theorem}[definition]{Theorem}
\newtheorem{lemma}[definition]{Lemma}

\newtheorem{remark}[definition]{Remark}

\def\N{{\mathbb N}}
\def\Z{{\mathbb Z}}
\def\R{{\mathbb R}}

\def\C{{\mathbb C}}

\newcommand{\lspan}{{\mathrm{span}}}
\newcommand{\supp}{{\mathrm{supp}}}
\newcommand{\lt}{{L^2(\R)}}

\newcommand{\re}{{\mathrm{Re} \,}}

\newcommand{\abs}[1]{\left\lvert #1\right\rvert}

\begin{document}

\title[Translation-based completeness on compact intervals]{Translation-based completeness on \\ compact intervals}

\author[Lukas Liehr]{Lukas Liehr}
\address{Faculty of Mathematics, University of Vienna, Oskar-Morgenstern-Platz 1, 1090 Vienna, Austria}
\email{lukas.liehr@univie.ac.at}


\date{\today}

\begin{abstract}
Given a compact interval $I \subseteq \mathbb{R}$, and a function $f$ that is a product of a nonzero polynomial with a Gaussian, it will be shown that the translates $\{ f(\cdot - \lambda) : \lambda \in \Lambda \}$ are complete in $C(I)$ if and only if the series of reciprocals of $\Lambda$ diverges. This extends a theorem in [R. A. Zalik, Trans. Amer. Math. Soc. 243, 299-308]. An additional characterization is obtained when $\Lambda$ is an arithmetic progression, and the generator $f$ constitutes a linear combination of translates of a function with sufficiently fast decay.
\end{abstract}

\maketitle

\section{Introduction and results}

Let $X$ be a Banach space of functions defined on a measurable set $I \subseteq \R$, such as $L^p(I)$ or $C(I)$, equipped with the usual norms. For a given $f : \R \to \C$ and a translation set $\Lambda \subseteq \R$, we define the system of $\Lambda$-translates of $f$ by
$$
\mathcal{T}(f,\Lambda) = \{ f(\cdot - \lambda) : \lambda \in \Lambda \}.
$$
The \emph{completeness problem of translates} investigates under what conditions on $X,f$, and $\Lambda$, the system $\mathcal{T}(f,\Lambda)$ is complete in $X$. A result by Wiener states that if $f \in L^1(\R)$ (resp. $f \in \lt$) then $\mathcal{T}(f,\R)$ is complete in $L^1(\R)$ (resp. $\lt$) if and only if the Fourier transform of $f$ does not vanish anywhere (resp. almost everywhere) \cite{wiener}. It was shown that such a characterization does not hold for $L^p(\R)$ with $1<p<2$ \cite{olevskii6,segal1944span}.

There has been significant interest in determining whether completeness of translates is possible when the translation set $\Lambda$ is a discrete set, such as the integers or a uniformly discrete set, i.e., a set with the property that there exists a constant $c>0$ such that $|\lambda-\lambda'| \geq c$ for all distinct $\lambda,\lambda' \in \Lambda$ \cite{olevskii1,olevskii2,abakumov2008cyclicity,lev2024completenessuniformlydiscretetranslates,olevskii1997completeness}.

In the present paper we are interested in the completeness problem of discrete translates in spaces of compactly supported functions. Precisely, we assume that $f \in C(\R)$ is a (globally) continuous function, $I \subseteq \R$ is a compact set, and $\Lambda \subseteq \R$ is a translation set. The question is whether for every $\varepsilon>0$ and every $g \in C(I)$ there exists $h \in \lspan(\mathcal{T}(f,\Lambda))$ such that
$$
\sup_{t \in I} |g(t)-h(t)| \leq \varepsilon.
$$

Inquiries of the above nature arise in the study of generalized periodic functions. Specifically, a function $f \in C(\R)$ is called \emph{mean-periodic} if the convolution equation
$$
f * \mu (x) \coloneqq \int_\R f(x-t) \, d\mu(t) = 0, \quad x \in \R,
$$
admits a non-trivial solution $\mu \in \mathcal{M}$, where $\mathcal{M}$ denotes the space of all compactly supported complex regular Borel measures on $\R$ \cite{schwartz,kahane}. By duality this is equivalent to the property that $\mathcal{T}(f,\R)$ is incomplete in $C(I)$ with $I = \supp(\mu)$.

Moreover, it has been recently shown that obtaining completeness results in $C(I)$ with respect to a discrete translation set $\Lambda$ implies solutions to the so-called \emph{uniqueness problem in phase retrieval} \cite{grohsliehr1,wellershoff}. The phase retrieval problem refers to a non-linear inverse problem which garnered considerable attention in recent times \cite{christ2023examples,bendory2022algebraic,jaming2023uniqueness}. A systematic connection between the uniqueness problem in phase retrieval and the completeness problem of translates was derived in \cite{grohs2022completeness}.

The completeness problem of $\mathcal{T}(f,\Lambda)$ in $C(I)$ for different choices of $f$ and $\Lambda$ was investigated by several authors:

\begin{itemize}
    \item \textbf{Akhiezer, Borwein-Erdélyi, Min.} Let $f(t)=\frac{1}{t}$ and let $\Lambda = \{ \lambda_n \} \subseteq \R \setminus [-1,1]$ be a sequence. A theorem of Akhiezer states that $\mathcal{T}(f,\Lambda)$ is complete in $C[-1,1]$, if and only if $\sum_n \sqrt{\lambda_n^2 - 1} = \infty$ \cite[p. 254, Problem 7]{achieser2013theory}. This theorem can be proved by establishing the existence of a so-called everywhere unbounded Bernstein inequality, as shown by Borwein and Erdély \cite{borwein1995dense,borwein2012polynomials}. Generalizations of this statement to rational functions beyond $f(t)=\frac{1}{t}$ are due to Min \cite{min1999denseness}.
    \item \textbf{Landau.} Suppose that $f \in L^1(\R) \setminus \{ 0 \}$ is the restriction to $\R$ of an entire function. Let $b>a>0$ and let $\Lambda$ be a zero-sequence of distinct numbers in $[-(b-a),b-a]$. A theorem of Landau states that $\mathcal{T}(f,\Lambda)$ is complete in $C[-a,a]$ \cite{LANDAU1972438}. 
    \item \textbf{Zalik.} Let $f(t)=e^{-ct^2}$ with $c>0$, let $\Lambda$ be a sequence of distinct real numbers, and let $I$ be a compact interval. A theorem of Zalik states that $\mathcal{T}(f,\Lambda)$ is complete in $L^2(I)$ if any only if $\sum_{\lambda \in \Lambda} \frac{1}{|\lambda|} = \infty$ \cite{Zalik}. An analogous statement with $L^2(I)$ replaced by $C(I)$ was shown by Wellershoff \cite{wellershoff}.
\end{itemize}

Motivated by Zalik's theorem, we consider functions of the form $f = q \varphi$, where $q$ is a polynomial and $\varphi$ is a Gaussian. Recent studies have shown that for such products, $\mathcal{T}(f, a \Z)$ is complete in $C(I)$, provided a density condition is met, i.e., the parameter $a$ is a sufficiently small positive constant depending on $I$ \cite{grohs2022completeness}.
The following result demonstrates that no density condition is required; instead, a Müntz-type condition is both necessary and sufficient. This leads to a characterization of completeness in the spirit of Zalik.

\begin{theorem}\label{thm:1}
    Let $q$ be a nonzero polynomial, let $\varphi(t)=e^{-ct^2}$ with $c>0$ be a Gaussian, let $I \subseteq \R$ be a compact interval, and let $\Lambda \subseteq \R$ be a sequence of distinct real numbers. Further, let $X$ be one of the spaces $C(I)$ or $L^p(I)$ with $p \in [1,\infty)$. Then the following two statements are equivalent.
    \begin{enumerate}
        \item $\mathcal T(q\varphi,\Lambda)$ is complete in $X$.
        \item The series of reciprocals of $\Lambda$ diverges,
        \begin{equation}
        \sum_{\substack{\lambda \in \Lambda \\ \lambda \neq 0}} \frac{1}{|\lambda|} = \infty.
    \end{equation}
    \end{enumerate}
\end{theorem}

Theorem \ref{thm:1} asserts that in Zalik's characterization of discrete translates, the Gaussian can be replaced by a product of a polynomial with a Gaussian. We now consider whether completeness is preserved when the generating function $f$ is replaced by a linear combination of shifts of $f$. The following statement shows that this is indeed the case if the generating function has a sufficiently fast decay and $\Lambda$ is an arithmetic progression, i.e., $\Lambda = a\Z + b$ for some $a, b \in \R$ with $a \neq 0$.

To state this precisely, we recall that a function $f : \R \to \C$ is said to have super-exponential decay if for every $\gamma \in \R$, the following condition holds:
\begin{equation}
\lim\limits_{|x| \to \infty} f(x)e^{\gamma |x|} = 0.
\end{equation}

\begin{theorem}\label{thm:2}
Let $a, b \in \R$ with $a \neq 0$, let $I \subseteq \R$ be a compact interval, and let $X$ be one of the spaces $C(I)$ or $L^p(I), \, p \in [1,\infty)$. If $f \in C(\R)$ has super-exponential decay, then the following statements are equivalent.
\begin{enumerate}
    \item The system $\mathcal{T}(f,a\Z+b)$ is complete in $X$.
    \item There exists $F \in \lspan(\mathcal{T}(f,a\Z+b))$ such that $\mathcal{T}(F,a\Z)$ is complete in $X$.
    \item For every $F \in \lspan(\mathcal{T}(f,a\Z+b)) \setminus \{ 0 \}$ it holds that $\mathcal{T}(F,a\Z)$ is complete in $X$.
\end{enumerate}
\end{theorem}

\section{Proof of Theorem \ref{thm:1}}

In order to prove Theorem \ref{thm:1} we require auxiliary results on holomorphic functions. Accordingly, we denote by $\mathcal{O}(\Omega)$ the space of all holomorphic functions on $\Omega \subseteq \C$. Further, $\mathbb H_+ \coloneqq \{ z \in \C : \re(z) > 0 \}$ and $\mathbb D \coloneqq \{ z \in \C : |z|<1 \}$ denote the open right half-plane and the open unit disk, respectively. The function
\begin{equation}\label{eq:phi}
    \phi(z) = \frac{z-1}{z+1}
\end{equation}
is a biholomorphic map from $\mathbb H_+$ to $\mathbb D$. 

We notice, that if $Z = \{ z_n \}_{n \in \N}$ are the zeros of an analytic function $F$, then $Z$ is understood as a sequence. In particular, if a zero $z'$ repeats $m$ times in the sequence $Z$, then $F$ has a zero of order $m$ at $z'$.

The following is a classical statement on zeros of holomorphic functions on the unit disc.

\begin{lemma}\label{lma:blaschke}
    Let $F \in \mathcal{O}(\mathbb D) \setminus \{ 0 \}$ be bounded with zeros $Z = \{ z_n \}_{n \in \N}$. Then the Blaschke condition is satisfied:
    $$
    \sum_{n \in \N} 1 - |z_n| < \infty.
    $$
\end{lemma}

The next result of Luxemburg and Korevaar concerns Fourier transforms of smooth functions supported on a compact interval \cite[Theorem 5.2]{LuxemburgKorevaar}.

\begin{theorem}\label{thm:lux_korevaar}
    Let $\{ \gamma_n \}_{n \in \N} \subseteq \C$ be a sequence of complex numbers, not containing infinitely many zeros. Further, let $I \subseteq \R$ be a compact interval. If the sequence $\{ \gamma_n \}_{n \in \N}$ satisfies
    $$
    \sum_{\substack{n \in \N \\  \gamma_n \neq 0}} \frac{1}{|\gamma_n|} < \infty,
    $$
    then there exists a function $g \in C^\infty(\R) \setminus \{ 0 \}$ with $\supp(g) \subseteq I$, such that
    $$
    G(z) \coloneqq \int_I e^{-izt}g(t) \, dt
    $$
    defines a non-trivial entire function with the property that $\{ \gamma_n \}_{n \in \N}$ is contained in the zero-sequence of $G$.
\end{theorem}

In addition, we require the following basic lemma, which, in the language of generalized periodic functions, says that every non-trivial integrable and continuous function is not mean-periodic \cite[Proposition 6.3]{pinkus}.

\begin{lemma}\label{lma:conv_el}
    Let $f \in L^1(\R) \setminus \{ 0 \}$ be continuous and let $\mu$ be a compactly supported complex regular Borel measure on $\R$. If $f * \mu = 0$ then $\mu=0$.
\end{lemma}

We are ready to prove Theorem \ref{thm:1}.

\begin{proof}[Proof of Theorem \ref{thm:1}]
In the following, we assume that the polynomial $q$ is of degree $d \in \N_0$ and is represented by
    \begin{equation}\label{eq:polynomial}
        q(t) = \sum_{j=0}^d \alpha_j t^j, \quad \alpha_j \in \C.
    \end{equation}

    \textbf{Step 1: Sufficiency.} Suppose that $\Lambda \subseteq \R$ is a sequence that satisfies the condition $\sum_{\lambda \in \Lambda \setminus \{ 0 \} } 1/|\lambda| = \infty$. We show that $\mathcal T(q\varphi,\Lambda)$ is complete in $C(I)$ which also yields completeness in $L^p(I)$ for $p \in [1,\infty)$.
    
    Define $\Lambda_1 = \{ \lambda \in \Lambda : \lambda > 0 \}$ and $\Lambda_2=\{ \lambda \in \Lambda : \lambda < 0 \}$. Then there exists $j \in \{ 1,2 \}$ such that
    \begin{equation}\label{eq:divergence_lambda_1}
        \sum_{\lambda \in \Lambda_j} \frac{1}{|\lambda|} = \infty.
    \end{equation}
    Suppose that $j = 1$ (the case $j=2$ can be treated analogously). By Riesz representation theorem \cite[Theorem 6.19]{rudin}, it suffices to show that if $\mu$ is a complex regular Borel measure on the interval $I$ such that
    \begin{equation}\label{eq:vanish}
        \int_I q(t-\lambda)\varphi(t-\lambda) \, d\mu(t) = 0, \quad \lambda \in \Lambda_1,
    \end{equation}
    then $\mu = 0$. Using the factorization $\varphi(t-\lambda) = e^{-ct^2}e^{2ct\lambda}e^{-c\lambda^2}$, we observe that equation \eqref{eq:vanish} is equivalent to the condition
    \begin{equation}
        F(\lambda) = 0, \quad \lambda \in \Lambda_1,
    \end{equation}
    where the entire function $F \in \mathcal{O}(\C)$ is defined by
    $$
    F(z) \coloneqq \int_I q(t-z) e^{2ctz} \, d\tilde{\mu}(t), \quad d \tilde \mu(t) \coloneqq e^{-ct^2}d\mu(t).
    $$
    The fact that $F$ defines an entire function follows from the theory of holomorphic integral transforms \cite[Chapter XII]{lang}.
    Now let $M \coloneqq \sup I$ be the supremum of $I$, and let $z=x+iy \in \mathbb H_+$. Since $x >0$, we obtain the estimate
    $$
    |F(z)| \leq \| \tilde \mu \| e^{2cM x} \sup_{t \in I} |q(t-z)|,
    $$
    where $\| \tilde \mu \|$ denotes the total variation norm of $\tilde \mu$. Since $q$ is of degree $d$ and since $I$ is compact, the map
    \begin{equation}\label{eq:zn1}
        z \mapsto \frac{\sup_{t \in I} |q(t-z)|}{|z+1|^{d+1}}
    \end{equation}
    is bounded on $\mathbb H_+$. Suppose that $C>0$ is an upper bound for the function in \eqref{eq:zn1} on $\mathbb H_+$. Consider functions $G$ and $H$, defined by
    $$
    G(z) \coloneqq H(z) F(z), \quad H(z) \coloneqq \frac{e^{-2cMz}}{(z+1)^{d+1}}.
    $$
    Then $G$ and $H$ are holomorphic on $\mathbb H_+$.
    Moreover, $G$ satisfies the upper bound
    \begin{equation}\label{eq:uniform_bound}
        |G(z)| = \frac{e^{-2cMx}}{|z+1|^{d+1}} |F(z)| \leq C \| \tilde \mu \|, \quad z = x+iy \in \mathbb{H}_+.
    \end{equation}
    Observe, that $H$ is zero-free on $\mathbb H_+$. Hence, since $F$ vanishes on $\Lambda_1$, so does $G$. Consequently, $G$ is an analytic function on $\mathbb H_+$ that vanishes on $\Lambda_1$ and that satisfies the uniform bound \eqref{eq:uniform_bound}.

With $\phi : \mathbb{H_+} \to \mathbb{D}$ defined as in \eqref{eq:phi}, it follows that $\tilde G \coloneqq G \circ \phi^{-1}$ is a bounded holomorphic function on $\mathbb D$. In addition, the divergence condition in \eqref{eq:divergence_lambda_1} implies that the zeros $Z(\tilde G) = \{ \phi(\lambda) : \lambda \in \Lambda_1 \}$ satisfy
    $$
    \sum_{\lambda \in \Lambda_1} 1-|\phi(\lambda)| = \infty.
    $$
    According to Lemma \ref{lma:blaschke}, $G$ must vanish identically.
    
    It follows that $F$ vanishes identically. This implies that equation \eqref{eq:vanish} holds not only for $\lambda \in \Lambda_1$ but for all $\lambda \in \R$. Therefore, the convolution of $\mu$ with the function $t \mapsto q(-t)\varphi(-t)$ vanishes identically. Lemma \ref{lma:conv_el} implies that $\mu =0$.

    \textbf{Step 2: Necessity.} Let $\Lambda \subseteq \R$ be a sequence such that $\sum_{\lambda \in \Lambda \setminus \{ 0 \}} 1/|\lambda| < \infty$. By Riesz representation theorem it suffices to construct a non-trivial complex regular Borel measure $\mu$ on $I$ such that
    $$
    \int_I q(t-\lambda) e^{-c(t-\lambda)^2} \, d\mu(t) = 0, \quad \lambda \in \Lambda.
    $$
    
    To do so, we start by defining a related sequence $\Gamma \subseteq \R$ as follows:
    $$
    \Gamma = \left \{ \underbrace{0,0, \dots, 0}_{d+1 \ \mathrm{times}}, \underbrace{2ic\lambda_1, 2ic\lambda_1, \dots, 2ic\lambda_1}_{d+1 \ \mathrm{times}}, \underbrace{2ic\lambda_2, 2ic\lambda_2, \dots, 2ic\lambda_2}_{d+1 \ \mathrm{times}},  \dots \right \}.
    $$
    By definition of $\Gamma$ it holds that
    $$
    \sum_{\substack{\gamma \in \Gamma \\ \gamma \neq 0}} \frac{1}{|\gamma|} = \frac{d+1}{2c} \sum_{\substack{\lambda \in \Lambda \\ \lambda \neq 0}} \frac{1}{|\lambda|} < \infty.
    $$
    According to Theorem \ref{thm:lux_korevaar}, there exists a function $g \in C^\infty(\R) \setminus \{ 0 \}$ with support in $I$ such that the (non-trivial) entire function
    $$
    G(z) = \int_I e^{-izt}g(t) \, dt
    $$
    vanishes on $\Gamma$. By construction of $\Gamma$, it follows that every point $2ic\lambda$ with $\lambda \in \Lambda$ is a zero of order at least $d+1$ of $G$. Notice, that if $\lambda \in \Lambda$ then
    $$
    G(2ic\lambda) = \int_I e^{2c\lambda t} g(t) \, dt.
    $$
    Using the definition of the polynomial $q$ as given in equation \eqref{eq:polynomial}, we obtain that for every $z \in \C$ it holds that
    \begin{equation}
        \begin{split}
            \int_I q(t-z) e^{2czt} g(t) \, dt & = \sum_{j=0}^d \sum_{k=0}^j \binom{j}{k} \alpha_j (-1)^{j-k} z^{j-k} \int_I t^j e^{2 c z t} g(t) \, dt \\
            & = \sum_{j=0}^d \sum_{k=0}^j \binom{j}{k} \alpha_j (-1)^{j-k} z^{j-k} (-i)^{-j} \left ( \frac{d^j}{dz^j} G \right )(2icz).
        \end{split}
    \end{equation}
    In the second equality we used the fact that we can differentiate under the integral sign.
    Since $G$ has zeros of order at least $d+1$ at every point $2ic\lambda$ where $\lambda \in \Lambda$, it follows that the right-hand side of the previous equation vanishes if $z = \lambda \in \Lambda$. In particular, the entire function
    $$
    z \mapsto e^{-cz^2} \int_I q(t-z) e^{2czt} e^{-ct^2} e^{ct^2} g(t) \, dt
    $$
    vanishes on $\Lambda$. By defining the measure
    $$
    d\mu(t) \coloneqq e^{ct^2} g(t) \, dt,
    $$
    and noting that $\varphi(t-\lambda) = e^{-ct^2}e^{2ct\lambda}e^{-c\lambda^2}$, we have
    $$
    \int_I q(t-\lambda)\varphi(t-\lambda) \, d\mu(t) = 0, \quad \lambda \in \Lambda.
    $$
    Since $\mu$ is a non-trivial complex regular Borel measure on $I$, it follows that $\mathcal T(q\varphi, \Lambda)$ is incomplete in $C(I)$.
    
    In order to show the related result for the spaces $L^p(I), \, p \in [1,\infty)$, we observe that the map $t \mapsto e^{ct^2} g(t)$ that determines the measure $\mu$ does not vanish identically and is an element of $L^{p'}(I)$ for all $p' \in (1,\infty]$. The statement follows from the duality $(L^p(I))^* = L^{p'}(I)$ where $p,p'$ are Hölder conjugate exponents.
\end{proof}

\begin{remark}
We remark that the proof of Zalik's theorem on the completeness of Gaussian translates \cite[Theorem 4]{Zalik} relies on the classical Weierstrass approximation theorem and Theorem \ref{thm:lux_korevaar}. In Zalik's proof, the statement of Luxemburg and Korevaar is directly applicable to the sequence $\Lambda$. In comparison, Theorem \ref{thm:1} provides a characterization for arbitrary products of polynomials with Gaussians using a purely complex-analytic proof.
\end{remark}

\section{Proof of Theorem \ref{thm:2}}

In the following we denote by $\mathbb{P}(\ell,n)$ the space of all complex polynomials in $n$ variables of degree at most $\ell$. Moreover, $T_s f \coloneqq f(\cdot -s)$ denotes the shift of a function $f : \R \to \C$ by $s \in \R$. The proof of Theorem \ref{thm:2} requires two technical lemmas. The first one reads follows.

\begin{lemma}\label{lemma:polynomials}
Let $f : \R \to \C$ be a function, let $a,b \in \R$ with $a \neq 0$, and let  $n \in \N$. Then for every $\ell \in \N$ there exist $n$ polynomials
$$
p_{\ell,1}, \dots, p_{\ell,n}
$$
with the following properties:
\begin{enumerate}
    \item\label{prop:i} For every $\ell \in \N$ and every $j \in \{ 1, \dots, n \}$ it holds that $p_{\ell,j} \in \mathbb{P}(\ell,n)$.
    \item\label{prop:ii} For every $j\in\left\lbrace 1,\dots, n\right\rbrace$, the largest coefficient in absolute value of $p_{\ell, j}$ is bounded from above by $2^{\ell-1}$.
    \item\label{prop:iii} For every coefficient vector $c = \left(c_1,\dots, c_n \right) \in \C^n$, it holds that the linear combination
    \begin{equation}\label{eq:Aell}
    A_\ell\coloneqq T_b f +  \sum\limits_{j=1}^n p_{\ell,j}(c) T_{a(\ell+j-1) + b}f
    \end{equation}
    is an element of $\lspan(\mathcal{T}(F, a\Z))$, where the function $F$ is defined by
$$
F \coloneqq T_b f + \sum\limits_{j=1}^n c_j T_{a j + b} f.
$$
\end{enumerate}
\end{lemma}

\begin{proof}
We prove the statement via induction over $\ell$. To do so, we first observe that $F \in \lspan(\mathcal{T}(F, a\Z))$. Hence, in the case where $\ell=1$, the statement holds true with polynomials $p_{1,1}, \dots, p_{1,n}$ defined according to
$$
p_{1,j}(x_1,\dots, x_n) \coloneqq x_{j}.
$$
Now suppose that the claim holds true for an arbitrary $\ell \in \N$. We show that it also holds for $\ell+1$.

Define
$$
q_j \coloneqq p_{\ell,j}(c_1,\dots,c_n) , \quad  j \in \{ 1,\dots, N\},
$$
that is, $q_j$ is the evaluation of the polynomial $p_{\ell,j}$ at $c$. With this notation, the map $A_\ell$ as defined in \eqref{eq:Aell} is given by
$$
A_\ell = T_b f + \left ( q_1 T_{a \ell+b} f + q_2 T_{a(\ell+1)+b} f + \dots + q_{N} T_{a(\ell+n-1)+b}f \right ).
$$
By induction hypothesis we have $A_\ell \in \lspan(\mathcal{T}(F, a\Z))$. Observe further that
$$
B_\ell \coloneqq q_1 \cdot T_{a \ell} F \in \lspan(\mathcal{T}(F, a\Z)).
$$
Since $\lspan(\mathcal{T}(F, a\Z))$ is a linear space, it holds that $A_\ell-B_\ell \in \lspan(\mathcal{T}(F, a\Z))$. The latter difference is given by
\begin{equation}\label{eq:A_minus_B}
        A_\ell-B_\ell = T_b f + \left ( \sum_{j=1}^{n-1} \left(q_{j+1} - q_1 c_{j}\right) T_{a(\ell+j)+b} f \right ) - q_1 c_n T_{a(\ell+n)+b}f.
\end{equation}
Now define polynomials $p_{\ell+1,1}, \dots, p_{\ell+1,n}$ by
\begin{equation}\label{eq:def_ploy}
\begin{split}
    p_{\ell+1,j}(x)& \coloneqq p_{\ell,j+1}(x) - p_{\ell,1}(x)x_j, \quad j \in \{1, \dots, n-1 \},  \\
    p_{\ell+1,n}(x) & \coloneqq -p_{\ell,1}(x)x_n.
\end{split}
\end{equation}
Then the difference $A_\ell-B_\ell$ can be written in terms of polynomial evaluations by
$$
A_\ell-B_\ell = T_b f + \sum_{j=1}^n p_{\ell+1,j}(c) T_{a(\ell+j)+b} f.
$$
Thus, by setting $A_{\ell+1} \coloneqq A_\ell-B_\ell$ we obtain the representation claimed in equation \eqref{eq:Aell}.

It remains to show that the polynomials satisfy properties \eqref{prop:i} and \eqref{prop:ii} given in the statement of the Lemma. 

By induction hypothesis, the polynomials $p_{\ell,j}$ are of degree $\ell$. According to the definition of the polynomials $p_{\ell+1,j}$ in \eqref{eq:def_ploy}, it holds that $p_{\ell+1,j} \in \mathbb{P}(\ell+1,n)$.

To show the remaining property, let $C_\ell$ be the largest coefficient in absolute value among all the polynomials $p_{\ell,j}$ where $j \in \{ 1, \dots, n \}$. It follows from the definition of the polynomials $p_{\ell+1,j}$ (as difference of at most two polynomials of the form $p_{\ell,j}$), that the largest coefficient in absolute value of $p_{\ell+1,j}$ does not exceed $2 C_\ell$.
The statement thus follows by induction.
\end{proof}

The second elementary lemma concerns super-exponentially decaying functions.

\begin{lemma}\label{lemma:decay}
Let $I \subseteq \R$ be a compact set and suppose that $f : \R \to \C$ is a measurable function with super-exponential decay. Then for every $a,b \in \R$ with $a \neq 0$ it holds that the function
$$
x \mapsto  \sup_{t \in I} |T_{ax+b}f(t)|
$$
has super-exponential decay.
\end{lemma}
\begin{proof}
    Let $\gamma \in \R$, let
    $$
    d \coloneqq \sup_{t \in I} |t-b|,
    $$
    and let $\varepsilon>0$. Since $f$ has super-exponential decay, there exists a constant $R(\varepsilon) > 0$ such that
    $$
    |f(t)|e^{\frac{\gamma}{|a|} t} \leq \varepsilon, \quad |t| \geq R(\varepsilon).
    $$
    For all $x \in \R$ with $|x| \geq \frac{R(\varepsilon)+d}{|a|}$ and every $t \in I$ it holds that $|t-(ax+b)| \geq R(\varepsilon)$. Accordingly, for such $x$ and $t$, we obtain
    \begin{equation}
        \begin{split}
            |T_{ax+b}f(t)|e^{\gamma|x|} & = |f(t-(ax+b))|\exp \left (\tfrac{\gamma}{|a|} \left |-(t-(ax+b))+t-b \right | \right ) \\
            & \leq |f(t-(ax+b))| \exp \left (\tfrac{\gamma}{|a|} \left | (t-(ax+b)) \right | \right )\exp\left (\tfrac{\gamma}{|a|}|t-b| \right ) \\
            & \leq \left ( \sup_{t \in I} \exp\left (\tfrac{\gamma}{|a|}|t-b| \right ) \right ) \varepsilon.
        \end{split}
    \end{equation}
    If $C$ denote the supremum on the right-hand side of the previous inequality, then
    $$
    \left ( \sup_{t \in I} |T_{ax+b}f(t)| \right ) e^{\gamma |x|} \leq C \varepsilon, \quad |x| \geq \frac{R(\varepsilon)+d}{|a|}.
    $$
    The statement follows from that.
\end{proof}

We are prepared to prove the second main result of the present paper.

\begin{proof}[Proof of Theorem \ref{thm:2}]
The implications $(3) \implies (2)$ and $(2) \implies (1)$ are clear. It remains to show that (1) implies (3).

Suppose that $\mathcal{T}(f,a\Z+b)$ is complete in $X$ and let $F \in \lspan(\mathcal{T}(f,a\Z+b)) \setminus \{ 0 \}$. We need to show that
$$
\mathrm{cl} \left ( \lspan(\mathcal T(F,a\Z)) \right ) = X,
$$
where $\mathrm{cl}(\cdot)$ denotes the closure with respect to the norm of $X$.

By definition, there exist $m_0,m \in \Z$ with $m_0 < m$ such that $F$ is represented by
$$
F = \sum_{k=m_0}^m d_k T_{a k +b}f,
$$
where $d_k \in \C$ and $k \in \{m_0,\dots, m \}$. Without loss of generality, we can assume that $d_{m_0}\neq 0$. Set $n = m-m_0$ and 
\begin{equation*}
    c_j = \frac{d_{m_0+j}}{d_{m_0}}, \quad j\in\left\lbrace 1,\dots, n \right\rbrace. 
\end{equation*}
Finally, define
$$
F_0 \coloneqq T_bf + \sum\limits_{j=1}^n c_j T_{a j + b} f = \tfrac{1}{d_{m_0}} \cdot T_{-a m_0}\, F.
$$
Then it holds that
\begin{equation}\label{eq:G_G0}
    \lspan(\mathcal{T}(F,a\Z)) = \lspan(\mathcal{T}(F_0,a\Z)).
\end{equation}

\textbf{Step 1: Application of Lemma \ref{lemma:polynomials}.}
By Lemma \ref{lemma:polynomials} and relation \eqref{eq:G_G0}, it follows that for every $\ell \in \N$, the function
\begin{equation}\label{eq:g_p_Q}
    A_\ell = T_bf + \sum\limits_{j=1}^n p_{\ell,j}(c) T_{a(\ell+j-1)+b} f,
\end{equation}
with polynomials $p_{\ell,j}$ defined according to Lemma \ref{lemma:polynomials}, belongs to the span of $\mathcal{T}(F,a\Z)$. Consequently, the maximal coefficient in absolute value among all polynomials $p_{\ell,j}$ is upper bounded by $2^{\ell-1}$. Notice that a polynomial of degree $\ell$ in $n$ variables has at most ${n+\ell} \choose \ell$ non-zero coefficients. The bound 
$$
{n \choose k} \leq \exp(k)\left(\frac{n}{k}\right)^k
$$
yields the estimate
\begin{equation}
    \begin{split}
        \abs{p_{\ell,j}(c)} & \leq {{n+\ell} \choose {\ell}} 2^{\ell-1} \max \{ \abs{c_1}, \dots, \abs{c_N} \}^{\ell} \\
        & \leq \exp(\ell) \left ( \frac{n+\ell}{\ell} \right )^{\ell} 2^{\ell-1} \max \{ \abs{c_1}, \dots, \abs{c_N} \}^{\ell}.
    \end{split}
\end{equation}
Since $\frac{n+\ell}{\ell} \to 1$ as $\ell \to \infty$, it follows that there exists a constant $C>0$ which is independent of $\ell$ and $j$, such that
\begin{equation}\label{bndd}
    \abs{p_{\ell,j}(c)} \leq C^\ell, \quad \ell \in \N, \quad j \in \{ 1, \dots, n \}.
\end{equation}

\textbf{Step 2: Taking the limit $\ell \to \infty$.}
Let $Q_\ell$ be the function defined by
$$
Q_\ell \coloneqq  
\sum\limits_{j=1}^n p_{\ell,j}(c) T_{a(\ell+j-1)+b}f.
$$
By equation \eqref{eq:g_p_Q}, we have $A_\ell = T_bf + Q_\ell \in \lspan(\mathcal{T}(F,a\Z))$ for every $\ell \in \N$. The bound \eqref{bndd} implies that
\begin{equation}\label{eq:Q_ell_uni}
    \sup_{t \in I} |Q_\ell(t)| \leq \sum_{j=1}^{n} C^\ell \sup_{t \in I} | T_{a(\ell+j-1)+b}f(t)|.
\end{equation}
Since $I$ is compact and since $f$ has super-exponential decay, Lemma \ref{lemma:decay} shows that each summand on the right-hand side of equation \eqref{eq:Q_ell_uni} converges to zero as $\ell \to \infty$. Consequently, $A_\ell \to T_b f$ uniformly on $I$. In particular, $A_\ell \to T_b f$ in the norm of $X$ (the uniform norm if $X=C(I)$ and the $L^p$-norm if $X=L^p(I)$). Since $A_\ell \in \lspan(\mathcal{T}(F,a\Z))$ for every $\ell \in \N$, it follows that
\begin{equation}\label{eq:g_in_closure}
    T_b f \in \mathrm{cl} \left ( \lspan(\mathcal{T}(F,a \Z)) \right ).
\end{equation}

\textbf{Step 3: Shift-invariance.}
The space $Y \coloneqq \mathrm{cl} \left (\lspan(\mathcal{T}(F,a\Z)) \right )$ is invariant under $a\Z$-shifts, that is, if $g \in Y$ then $T_{aj} g \in Y$ for all $j \in \Z$. Together with equation \eqref{eq:g_in_closure} it follows that $T_{aj+b}f \in Y$ for all $j \in \Z$. Thus,
$$
\lspan(\mathcal{T}(f,a\Z+b)) \subseteq  Y \subseteq X.
$$
Since $\mathcal{T}(f,a\Z+b)$ was assumed to be complete in $X$, it follows that $Y$ is complete in $X$.
\end{proof}

\bibliographystyle{abbrv}
\bibliography{bibfile}

\end{document}